\documentclass[a4paper]{article}

\usepackage{natbib}

\usepackage{amsfonts}
\usepackage{amsmath}
\usepackage{amssymb}
\usepackage{amsthm}
\usepackage{enumerate}
\usepackage[T1]{fontenc} 
\usepackage{verbatim}

\def\è{\`e}

\def\CC{{\mathbb{C}}}
\def\QQ{{\mathbb{Q}}}
\def\RR{{\mathbb{R}}}
\def\ZZ{{\mathbb{Z}}}

\def\KK{{\mathbb{K}}}
\def\FF{{\mathbb{F}}}

\def\LL{{\mathbb{L}}}

\def\tdeg{{\textrm{tdeg }}}

\newtheorem{Thm}{Theorem}
\newtheorem{prop}[Thm]{Proposition}
\newtheorem{lem}[Thm]{Lemma}
\newtheorem{cor}[Thm]{Corollary}
\newtheorem{rem}[Thm]{Remark}

\theoremstyle{definition}
\newtheorem{defn}[Thm]{Definition}
\newtheorem{exmp}[Thm]{Example}

\title{Modular Las Vegas Algorithms  for Polynomial Absolute Factorization}
\date{\ }

\author{Cristina Bertone$\phantom{,}^{a,b}$, Guillaume Ch\`eze$\phantom{,}^{c}$, Andr\'e Galligo$\phantom{,}^{a}$\\
\ \\
\small{$\phantom{,}^{a}$Laboratoire J.-A. Dieudonn\'e, Universit\'e de Nice - Sophia Antipolis, France}\\
  \small{$\phantom{,}^{b}$Dipartimento di Matematica, Universit\`a degli Studi di Torino, Italy}\\
\small{$\phantom{,}^{c}$Institut de Math\'ematiques de Toulouse, Universit\'e Paul Sabatier Toulouse 3, France}
}

\begin{document}


\maketitle

\begin{abstract}
Let $f(X,Y) \in \ZZ[X,Y]$ be an irreducible polynomial over $\QQ$. We  give a Las Vegas absolute  irreducibility test based on a property of the Newton polytope of $f$, or more precisely, of $f$ modulo some prime integer $p$. The
same idea of choosing a $p$ satisfying some prescribed properties together with $LLL$ is used to provide a new strategy for absolute factorization of $f(X,Y)$. We present our approach in the bivariate case but the techniques extend to the multivariate case.
Maple computations show that it is efficient and promising as we are able to factorize some polynomials of degree up to 400.
\end{abstract}



\section*{Introduction}
Kaltofen's survey papers \citep{K2} related the early success story of polynomial factorization. Since then, crucial progresses  
have been achieved : algorithms developed and implemented by Van Hoeij and his co-workers in the univariate case \citep{Bel}, by Gao and his co-workers (see for instance \citet{Gao1}), then by Lecerf and his co-workers in the multivariate case (\citet{BLSSW}, \citet{Lec07}). \citet{Che1}, \citet{CheLec} and \citet{Lec07} also improved drastically the multivariate absolute factorization (i.e. with coefficients in the algebraic closure): they  produced an algorithm with the best known arithmetic complexity.
Even if the situation evolved  rapidly, there is still room for improvements and new points of view.

 Here, we focus on absolute factorization of rationally irreducible polynomials with integer coefficients (see \citet{CG1}, \citet{Rupprecht2}, \citet{svw2} and the references therein).
For such polynomials, the best current algorithm and implementation is Ch\`eze's  (\citet{Che}, \citet{Che1}) presented at Issac'04, it is based on semi-numerical computation, uses LLL and is implemented in Magma. It can factorize  polynomials of high degrees, up to 200. One of the challenges is to improve its capabilities at least in certain situations. 

We propose yet another strategy and algorithm to deal with (multivariate) absolute irreducibility test and factorization.
This article will present a simple, but very efficient,  irreducibility test. Then we extend our strategy  to get a  factorization algorithm based on modular computations, Hensel liftings and recognition of algebraic numbers via  $p-$adic approximation using $LLL$ (as explained in \citet{VZG}).\\
Our absolute factorization algorithm can be viewed as a drastic improvement of the classical algorithm TKTD (see \citet{TrDv}, \citet{K1}, \citet{Trager2} and Section \ref{terzasez}). Indeed, we replace the computations
in an algebraic extension of $\QQ$ of degree $n$, the degree of the input polynomial, by computations in an extension of the minimal degree $s$, the number of factors of the input polynomial.

We made a preliminary implementation in Maple and computed several examples. It is very promising  as it is fast and able to compute the researched algebraic extension for high degree polynomials
(more than degree 200, see last section). The bottleneck of the procedure is now the final $x$-adic Hensel lifting, but we may avoid this problem with a parallel version of our algorithm, as explained in Section \ref{parallVers}.

In other words,  our approach improve the practical complexity of absolute factorization of polynomials with integer coefficients.

\subsection*{Notations}
\noindent $\KK$ is a perfect field, $\overline{\KK}$ is an algebraic closure of $\KK$.\\
$\FF_p=\ZZ/p\ZZ$ is the finite field with $p$ elements, where $p$ is a prime integer.\\
$\tdeg f$ is the total degree of the polynomial $f$.

\section{Absolute irreducibility test and Newton Polytope}\label{secone}

Any implementation of an absolute factorization algorithm needs to  first check  if the polynomial is ``trivially'' absolutely irreducible. That is to say, test quickly a sufficient condition on $f$: when the test says yes, then $f$ is absolutely irreducible and the factorization algorithm can be spared. The test should be fast and should, in  ``most'' cases (i.e. with a good probability)  say yes when the polynomial $f$ is irreducible. For instance, for polynomials of degree 100, one might expect that such a test runs 100 time faster than a good general factorization algorithm. This is indeed the case for the test presented in this section: for a polynomial of degree 100,  absolute factorization algorithms (e.g. the ones in  \citet{Che} and \citet{CheLec}) require 20 seconds to decide irreducibility  while our test answers after only  0.07 seconds.\par

 The absolute irreducibility test presented in this article is based on properties of the Newton polytope of a polynomial that we now review.

\begin{defn}
Let $f(X,Y)=\sum_{i,j}c_{i,j}X^iY^j \in \KK[X,Y]$. The Newton polytope of $f$, denoted by $P_f$, is the convex hull in $\RR^2$ of all the points $(i,j)$ with $c_{i,j} \neq 0$.

A point $(i,j)$ is a \emph{vertex} of $P_f$ if it is not on the line
segment of any other two points of the polytope. \qed
\end{defn}

Remember that a polytope is the convex hull of its vertices.

We refer to \citet{Gao2}  for basic results on absolute irreducibility and Newton polytopes and also for an interesting short history which goes back to the famous Eisenstein criterion.

\begin{defn}
Denote by $(i_1,j_1)$, \ldots, $(i_l,j_l) \in \ZZ^2$ the vertices of $P_f$.
We say that condition $(C)$ is satisfied when $\gcd(i_1,j_1,\ldots,i_l,j_l)=1$. \qed
\end{defn}

The aim of this section is to prove the following criterion.

\begin{prop}[Absolute irreducibility criterion]\label{test}\-

Let $f(X,Y)$ be an irreducible polynomial in $\KK[X,Y]$. 
If   condition $(C)$ is satisfied then $f$ is absolutely irreducible.
\end{prop}

Our statement in Proposition \ref{test} bears similarities with one of  Gao's result \citep{Gao2}; but it differs since Gao assumed that $P_f$ should be contained in a triangle when we assume that $f$ is irreducible in $\KK[X,Y]$. Although, our condition seems a strong theoretical hypothesis, in practice we can check it very quickly thanks to the algorithms developed in \citet{BLSSW} and \citet{Lec04}. The advantage of our criterion is that it applies  to a larger variety of polytopes.

 We first recall an important lemma about absolute factorization of (rationally) irreducible polynomials.

\begin{lem}\label{lem_fonda}
Let $f \in \KK[X,Y]$ be an irreducible polynomial in $\KK[X,Y]$, monic in $Y$:
$$f(X,Y)=Y^n+\sum_{k=0}^{n-1}\sum_{i+j=k} a_{i,j} X^i Y^j.$$ 
Let  $ f=f_1 \cdots f_s$ be the monic factorization of $f$ by irreducible polynomials $f_l$ in $\overline{\KK}[X,Y]$. Denote by  $\LL=\KK[\alpha]$ the extension of $\KK$ generated by all the coefficients of $f_1$. Then each $f_l$ can be written: 
\begin{equation}\label{bfattass}
f_l(X,Y)=Y^m+\sum_{k=0}^{m-1}\sum_{i+j=k} a_{i,j}^{(l)}X^i Y^j= Y^m+\sum_{k=0}^{m-1}\sum_{i+j=k} b_{i,j}(\alpha_l) X^i Y^j,
\end{equation}
where $b_{i,j} \in \KK[Z]$, $\deg_Z (b_{i,j}) < s$  and where $\alpha_1,\ldots,\alpha_s$ are the different conjugates over $\KK$ of $\alpha = \alpha_1$. \qed
\end{lem}

See \citep[Lemma 2.2]{Rupprecht2} for a proof.

As a corollary the number of absolute factors is equal to $[\LL:\KK]$.

In order to prove Proposition \ref{test}, we introduce the Minkowski sum  and its pro\-per\-ties concerning polytopes.

\begin{defn}
If $A_1$ and $A_2$ are two subsets of the vector space $\RR^n$, we define their \emph{Minkowski sum} as
\[
A_1+A_2=\{a_1+a_2\, | \, a_1\in A_1,a_2 \in A_2\}.\qed
\]
\end{defn}

\begin{lem}[Ostrowski] Let $f, g,h \in \KK[X_1,X_2,\dots, X_n]$ with $f = gh$. Then
$P_f =P_g+P_h$.
\end{lem}
\begin{proof} 
See \citet{Ost}. 
\end{proof}

In particular \citep{Schneider}, if we sum up $s$ times the same convex polytope $A$, then we have that
\[
\underbrace{A+\cdots+A}_{s-times}=s\cdot A,
\]
where $s\cdot A=\{s\cdot v \, | \, v \in A\}$. Furthermore the vertices $\{v_1,\dots, v_l\}$ of $s\cdot A$ are exactly $v_i=s\cdot w_i$, where $\{w_1,\dots,w_l\}$ is the set of vertices of $A$.

\medskip

We  now consider the  irreducible polynomial $f\in \KK[X,Y]$ and its absolute factors $f_1,\dots,f_s\in \overline \KK[X,Y]$. Observe that thanks to Lemma \ref{lem_fonda}, we have that $P_{f_i}=P_{f_j}$ for every couple of indexes $i,j \in \{1,\dots,s\}$.

We can then easily prove Proposition \ref{test}.
\begin{proof}
Suppose that  $f$ is not absolutely irreducible. 
Let $f_1,\dots, f_s$ be the absolute factors of $f$. For what concerns the Newton polytopes, we have that
\[
P_f=P_{f_1}+\cdots+P_{f_s}=s\cdot P_{f_1}.
\]
Suppose in particular that the vertices of $P_{f_1}$ are $\{(i_1,j_1),\dots,(i_l,j_l)\}$. Then we have that the vertices of $P_f$ are $\{(s\cdot i_1,s\cdot j_1),\dots,(s\cdot i_l,s\cdot j_l)\}$. But then condition $(C)$ is not satisfied.
\end{proof}

\begin{cor}
The number of absolute irreducible factors of a ra\-tio\-nal\-ly irreducible polynomial 
 $f(X,Y) \in \KK[X,Y]$ divides $\gcd(i_1,j_1,\ldots,i_l,j_l)$.
\end{cor}
\begin{proof}
This is a consequence of the proof of Proposition \ref{test}.
\end{proof}
As all the arguments we used in this section extend to Newton polytopes in any number of variables we get:

\begin{cor}
Proposition \ref{test} holds for a polynomial ring with any number of variables. \qed
\end{cor}

\section{Evaluation of our irreducibility criterion}
In Proposition \ref{test}, we established the validity of our criterion. In this section we address the natural question:
does condition $(C)$ happens frequently ?

When the polynomial $f$ is dense, then the coordinates of the vertices of $P_f$ are $(0,0)$, $(n,0)$, $(0,n)$, thus condition $(C)$ is not satisfied and we cannot apply our test.
However when $f$ is sparse, in  ``most'' cases, the Newton polytope is not the triangle of the previous situation and  a direct use of Proposition  \ref{test} can quickly detect if $f$ is absolutely irreducible.

We first provide time tables and statistic evidences of the efficiency of our criterion applied  to a sparse polynomial   $f(X,Y) \in \ZZ[X,Y]$. Then we consider its application to dense polynomials. In that case, modular computations are used to force a sparsity  condition on a reduced polynomial modulo some prime $p$.
\subsection{ Statistics for a direct use of the test for sparse polynomials}

To check the previous claim, we have constructed randomly $1000$ polynomials of total degree $n$ and applied our test. Our test is implemented in Magma and available at: http://www.math.univ-toulouse.fr/$\sim$cheze/

The following table presents the obtained statistical results.

The entries are the degree $n$ and  a sparsity indicator $Prop$. When its value is $Prop=1$ (respectively $Prop=2$), each polynomial has about $n(n+1)/4$ (respectively $n(n+1)/6$) non-zero coefficients randomly chosen in $[-10^{12};10^{12}]$ and $n(n+1)/4$ (respectively $n(n+1)/3$)   coefficients randomly chosen equal to zero. The outputs are: 
the number $Success$ of absolute irreducible polynomials detected by our test, and  the average running time  $T_{av}$  (in second).

\begin{center}\label{tableau_test_polytop}
\begin{tabular}{|c|c|c|c|c|c|}
\hline
$n$ & $Prop$  & $Success$ & $T_{av}$\\
\hline
\hline
\hline
50 & 1& 819 & 0.0134\\
\hline
50 & 2 & 943 & 0.0122\\
\hline
100 & 1 & 832 & 0.0787\\
\hline
200 & 1 &  849 & 0.6023 \\
\hline
200 & 2 & 948  & 0.4432 \\
\hline
\end{tabular}\\
\end{center}

\vspace{0.1cm}

 This table shows that our test is well suited for sparse polynomials.

\subsection{Irreducibility test with modular computations}

Our aim is to construct  a sparse polynomial associated to  a dense polynomial, ``breaking'' its Newton polytope. 
For that purpose, we recall an easy corollary of Noether's irreducibility theorem. For a statement and some results about Noether's irreducibility theorem see e.g. \cite{Kaltof95}.

\begin{prop}\label{prop_mod_p}
Let $f(X,Y) \in \ZZ[X,Y]$ and $\overline{f}(X,Y)=f \mod p$, $\overline f \in \FF_p[X,Y]$.\\
 If \mbox{$\tdeg (f)= \tdeg (\overline{f})$} and $\overline{f}$ is absolutely irreducible, then $f$ is absolutely irreducible. \qed
\end{prop}

Now, even if $f$ is dense, the idea is to choose $p$ in order to force   $\overline{f}$ to be sparse. Then we  apply the test to $\overline{f}$ instead of applying it to $f$.

Let $a_1, \ldots, a_r$ be the coefficients corresponding to the vertices of $P_f$ and $L=[p_1, \ldots, p_l]$ be the list of the primes dividing at least one of the $a_i$. Remark that:
$$ \forall p_i \in L, \, P_f \neq P_{\,f \hspace{-0.2cm} \mod p_i}.$$
Thus even when $f$ is dense, if the coefficients $a_1, \ldots, a_r$ are not all equal to $1$, we can get polynomials $f \mod p_i$ such that $P_{\,\,f \mod p_i}$ is not the triangle with vertices $(0,0)$, $(0,n)$, $(n,0)$. In Section  \ref{test_irred_mod_p_chag_var}, we will see  that a linear change of coordinates permits to deal with the remaining case.

\emph{Example:} $f(X,Y)=Y^3+X^3+5X^2+3Y+2$. Figure \ref{fefmodp} clearly illustrates the effect of a reduction modulo $p=2$.
\begin{center}
\begin{figure}[h]
\begin{minipage}[c]{0.45\linewidth}
\centering
\setlength{\unitlength}{1cm}
\begin{picture}(5,5)

\put(1,1){\vector(1,0){4}}
\put(1,1){\vector(0,1){4}}

\put(2,0.925){\line(0,1){0.15}}
\put(3,0.925){\line(0,1){0.15}}
\put(4,0.925){\line(0,1){0.15}}
\put(0.925,2){\line(1,0){0.15}}
\put(0.925,3){\line(1,0){0.15}}
\put(0.925,4){\line(1,0){0.15}}

\put(0.5,0.5){{$0$}}
\put(2,0.5){{$1$}}
\put(3,0.5){{$2$}}
\put(4,0.5){{$3$}}
\put(0.5,2){{$1$}}
\put(0.5,3){{$2$}}
\put(0.5,4){{$3$}}
\put(5,0.5){{$X$}}
\put(0.5,5){{$Y$}}

\put(4,1){\circle*{.2}}
\put(1,4){\circle*{.2}}
\put(1,1){\circle*{.2}}

\thicklines\put(1,4){\line(1,-1){3}}
\thicklines\put(1,1){\line(1,0){3}}
\thicklines\put(1,1){\line(0,1){3}}
\end{picture}
\end{minipage}
\hfill
\begin{minipage}[c]{0.45\linewidth}
\centering
\setlength{\unitlength}{1cm}
\begin{picture}(5,5)
\put(1,1){\vector(1,0){4}}
\put(1,1){\vector(0,1){4}}

\put(2,0.925){\line(0,1){0.15}}
\put(3,0.925){\line(0,1){0.15}}
\put(4,0.925){\line(0,1){0.15}}
\put(0.925,2){\line(1,0){0.15}}
\put(0.925,3){\line(1,0){0.15}}
\put(0.925,4){\line(1,0){0.15}}

\put(0.5,0.5){{$0$}}
\put(2,0.5){{$1$}}
\put(3,0.5){{$2$}}
\put(4,0.5){{$3$}}
\put(0.5,2){{$1$}}
\put(0.5,3){{$2$}}
\put(0.5,4){{$3$}}
\put(5,0.5){{$X$}}
\put(0.5,5){{$Y$}}

\put(4,1){\circle*{.2}}
\put(1,4){\circle*{.2}}
\put(1,2){\circle*{.2}}
\put(3,1){\circle*{.2}}

\thicklines\put(1,4){\line(1,-1){3}}
\thicklines\put(1,2){\line(2,-1){2}}
\thicklines\put(1,2){\line(0,1){2}}
\thicklines\put(3,1){\line(1,0){1}}
\end{picture}
\end{minipage}
\caption{Newton polytopes of $f$ and $f \mod 2$}
\label{fefmodp}
\end{figure}
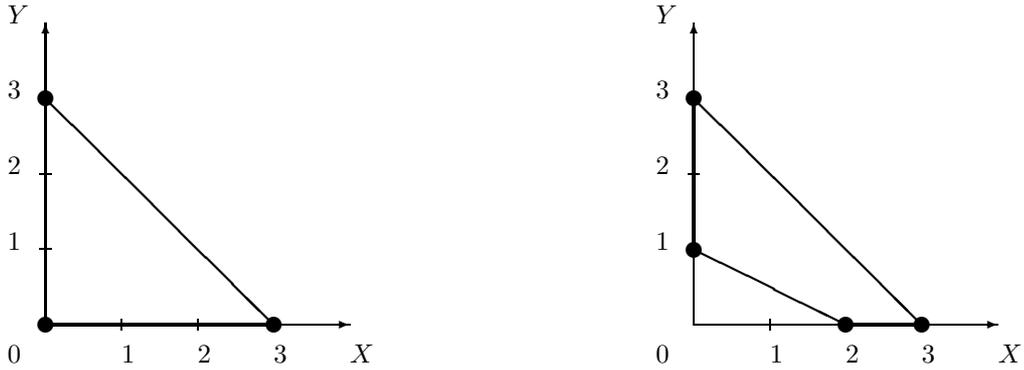
\end{center}

\noindent Therefore, thanks to  Proposition \ref{test} and  Proposition \ref{prop_mod_p}, absolute irreducibility can be tested with a Las Vegas strategy (i.e. the output of the algorithm is  \emph{always} correct). However the output can be ``I don't know''. 
More precisely:\\
For each $p \in L$, test the absolute irreducibility of $\overline{f} \in\FF_p[X,Y]$ with Proposition \ref{test}, and conclude with Proposition \ref{prop_mod_p}.

\medskip

\textbf{\textsc{Newton-polytop-mod algorithm}}\\
\textsc{Inputs:} $f(X,Y) \in \ZZ[X,Y]$, irreducible in $\QQ[X,Y]$.\\
\textsc{Outputs:} ``$f$ is absolutely irreducible'' or ``I don't know''.

\begin{enumerate}
\item Compute $P_f$ and the list $L$  of the primes dividing a coefficient cor\-re\-spon\-ding to a vertex of $P_f$.$\;$
Initialize $\;$ test:=false: $i:=1$:
\item While(test=false) and ($i \leq |L|$) do 
$p:=L[i]$;\\
If $tdeg( f \mod p)= tdeg(f)$ then\\
\hspace*{1cm}Compute $P_{f \mod p}$.\\
\hspace*{1cm}If $f \mod p$ satisfies condition $(C)$ then\\
\hspace*{1.5cm}If $f \mod p$ is irreducible in $\FF_p[X,Y]$ then test:=true; \mbox{End If};\\
\hspace*{1cm}End If; $\;$
End If;$\;$
$i:=i+1$ $\;$
End While:
\item If (test = true) then return ``$f$ is absolutely irreducible''  else return ``I don't know''  End If:
\end{enumerate}

The following table shows  that this algorithm is quite efficient. We constructed $1000$ polynomials in $\ZZ[X,Y]$ of total degree $n$, with random integer coefficients in $[-10^{12};10^{12}]$. All these polynomials are dense. For each polynomial we test its absolute irreducibility with the previous algorithm. $Success$ is the number of absolute irreducible polynomials detected with this algorithm. $T_{av}$ (respectively $T_{max}$, $T_{min}$) is the average (respectively maximum,  minimum) timing in second to perform one test.

\begin{center}\label{tableau_test_polytop_mod}
\begin{tabular}{|c|c|c|c|c|c|c|} 
\hline
$n$ & $Success$ & $T_{av}$ & $T_{max}$ & $T_{min}$ \\
\hline
\hline
10 & 1000 & 0.0041 & 0.33 & 0\\
\hline
30 & 1000 & 0.0113 & 0.56 & 0\\
\hline
50 & 1000 & 0.0252 & 0.59 & 0.009\\
\hline
100 & 1000 & 0.1552 & 0.66 & 0.081\\
\hline
200  & 1000 & 1.7579 & 3.22 & 0.701\\
\hline
\end{tabular}\\
\end{center}

\subsection{Modular computations and change of coordinates}\label{test_irred_mod_p_chag_var}

A last task is to  deal with polynomials whose coefficients are $0$, $1$ or $-1$
like $f(X,Y)=X^n+Y^n+1$, because in that case the Newton polytope gives no information, even when one looks at the modular reduction $f \mod p$. The natural strategy is to perform  a linear change of coordinates in order to obtain, after reduction, a polynomial satisfying  condition $(C)$. This is applied in the next algorithm.

Modular computation is performed in $\FF_p$ where $p$ is a prime between 2 and some value, here fixed to 101. 

\medskip

 \textbf{\textsc{Newton-Polytop-mod-chg-var algorithm}}\\
\textsc{Input:} $f(X,Y) \in \ZZ[X,Y]$,  irreducible in $\QQ[X,Y]$.\\
\textsc{Output:} ``$f$ is absolutely irreducible'' or ``I don't know''.

 For each $p$ prime between 2 and 101 do:\\
\hspace*{.5cm} For $(a,b) \in \FF_p^2$ do\\
\hspace*{1cm} $f_{a,b}(X,Y)=f(X+a,Y+b) \mod p$;\\
\hspace*{1cm} If $\tdeg(f_{a,b})=\tdeg(f)$ then\\
\hspace*{1.5cm} If $f_{a,b}$ satisfies condition $(C)$ then\\
\hspace*{2cm} If $f_{a,b}$ is irreducible in $\FF_p[X,Y]$ then return 

``$f$ is absolutely irreducible'';

\hspace*{1.5cm} End If; End If; End If;End If; End For; End For;

\noindent
Return ``I don't know''.

\medskip

This algorithm generalizes a test given by  \citet{Ragot2} based on the following classical property. 

{\bf Fact:} \emph{Let $f(X,Y) \in \KK[X,Y]$ be an irreducible polynomial in $\KK[X,Y]$. If there exists $(a,b) \in \KK^2$ such that $f(a,b)=0$ and $\dfrac{\partial f}{\partial X}(a,b) \neq 0$ or $\dfrac{\partial f}{\partial Y}(a,b) \neq 0$, then $f$ is absolutely irreducible.}

Ragot's algorithm tests if $f \mod p$ has a simple root in $\FF_p$. Remark that $f$ has a simple root if and only if  after a linear change of coordinates, which brings this root at the origin, the Newton polytope of $f$ has at least one of the points $(1,0)$ and $(0,1)$ as vertex, while $(0,0)$ is not a vertex.

\noindent In that case, condition $(C)$ is  satisfied; thus  Ragot's test is weaker than our test. \\
At http://www.mip.ups-tlse.fr/$\sim$cheze/, we listed an example of polynomial for which  absolute irreducibility is immediately detected by our algorithm reducing modulo  $p=2$, while  Ragot's test needs to reduce and check iteratively for all primes until   $p=73$.

\medskip

Let us remark that thanks to the following proposition, for $p \geq (n-1)^4$  our probabilistic test becomes deterministic.
\begin{prop}[\citet{Ragot}, Prop. 4.4.3 page 79]
Let $f(X,Y) \in \FF_p[X,Y]$ be an absolute irreducible polynomial of total degree $n$. If $p \geq (n-1)^4$ then $f$ has simple roots in $\FF_p$. \qed
\end{prop}

 Indeed, if we have a simple root then after a change of coordinates we get a polynomial satisfying Ragot's test and thus satisfying condition $(C)$. 
However, in practice, a probabilistic approach with a small prime is much faster.

\medskip

We only considered the case of integer polynomials, however our tests can be extended to the case of polynomials with coefficients in a commutative ring. In this case, the computation modulo a prime number will be replaced by a computation modulo a prime ideal.
The algorithms can also be extended to the case of polynomials with $N$ variables, in which case the probability of success will increase with $N$. Indeed, there are more chances to obtain a $\gcd$ equal to 1 with more coordinates.

\section{A toolbox for an absolute factorization algorithm}\label{terzasez}

We aim to build a factorization algorithm by extending the analysis and strategy developed for the previous irreducibility test. We keep the notations introduced in Section \ref{secone} and specially in Lemma \ref{lem_fonda}.
A main  task is to describe an algebraic extension $\LL=\QQ(\alpha)$ of $\QQ$ which contains the coefficients of a factor $f_1$ of $f$. 

This kind of strategy was already developed in the TKTD algorithm; TKTD is an acronym for Trager/ Kaltofen/Traverso/Dvornicich, (see \citet{TrDv}, \citet{K1} and  \citet{Trager2}). The result of the TKTD algorithm is an algebraic extension $\LL$ in which $f(X,Y)$ factors. Usually this extension is too big, that is to say: the degree extension of $\LL$ is not minimal.

We aim to reach the same goal, obtain an algebraic extension in which $f(X,Y)$ is reducible,  but the extension we will find is smaller, in fact minimal, and so more suitable for the computation of the factorization.


\subsection{Algebraic extensions and primitive elements}

We can describe the extension $\LL$  of $\QQ$ with a primitive element. Let us see that, generically, $\LL=\QQ[f_1(x_0,y_0)]$.

\begin{lem}\label{cst_primitif}
Let $f(X,Y)  \in \ZZ[X,Y]$ be a rationally irreducible polynomial (i.e. over $\QQ$) of degree $n$. Let $f_1(X,Y)$ be an absolute irreducible factor of $f(X,Y)$, $\deg f_1(X,Y)=m$.\\
For almost all $(x_0,y_0) \in \ZZ^2$ we have $\LL=\QQ(f_1(x_0,y_0))$.\\
More precisely, the following estimate on the probability holds:
$$\mathcal{P}\Big( \{ (x_0,y_0) \in S^2 \mid \LL= \QQ(f_1(x_0,y_0)) \} \Big) \geq 1-\dfrac {n(s-1)}{2\vert S \vert} \quad \text{ with } s:=n/m,$$
where $S$ is a finite subset of $\ZZ$.
\end{lem}

\begin{proof}
We denote by  $a_{i,j}$ the coefficients of $f_1$, so  $\LL= \QQ(a_{i,j})$. Let $\sigma_l$, $(1 \leq l \leq s)$ be  $s$ independent $\QQ$-homomorphisms from $\LL$ to $\CC$. 

Hence we have: 
\begin{equation*}\tag{$\ast$}
\forall u \neq v, \text{ there exists }(i,j)\text{ such that }\sigma_u(a_{i,j}) \neq \sigma_{v}(a_{i,j}).
\end{equation*}
We consider $D(X,Y)=\prod_{u\neq v}\Big( \sum_{i,j} \big(\sigma_u-\sigma_v\big)(a_{i,j}) X^iY^j \Big)$.

Property $(\ast)$ implies that $D(X,Y) \neq 0$. Then there exists $(x_0,y_0) \in \ZZ^2$ such that $D(x_0,y_0) \neq 0$. This means: for all $u \neq v$, $\sigma_u \big(  f_1(x_0,y_0)\big) \neq \sigma_v \big(f_1(x_0,y_0)\big)$. Thus $f_1(x_0,y_0)$ is a primitive element of $\LL$ and this gives the desired result.

The probability statement is a direct consequence of Zippel-Schwartz's lemma, applied to $D(X,Y)$, whose degree is bounded by $(ms(s-1))/2=(n(s-1))/2$.
\end{proof}

Remark that the polynomial $D(X,Y)$ appearing in the previous proof is also the di\-scri\-mi\-nant, with respect to $Z$, of the 3-variate polynomial  $F(X,Y,Z)=\prod_{j} (Z-f_j(X,Y))$.
$F$ has coefficients in $\ZZ$ because its coefficients are invariant when we permute the $f_j$.

\subsection{Number fields and p-adic numbers}
\begin{lem}\label{rootinQp}
Let $M(T) \in \ZZ[T]$ be a polynomial and $p$ a prime number such that $p$ divides $M(0)$ and $p>\deg(M)$.\\
Then there exists a root in $\QQ_p$ of $M(T)$, considered as a polynomial in $\QQ_p[T]$.
\end{lem}
 
This lemma allows us to consider a number field $\QQ(\alpha)$ as a subfield of $\QQ_p$, for a well-chosen prime $p$. Indeed, if $q(T)$ is the minimal polynomial of $\alpha$, then with a big enough integer $c$ we can find a prime number $p$ such that the polynomial $q(T+c)$ satisfies the hypothesis of Lemma \ref{rootinQp}. Thus we can consider $\alpha+c$ in $\QQ_p$, then $\QQ(\alpha)\subset \QQ_p$. During our algorithm we are going to factorize $f(X,Y) \mod p$. We can consider this factorization as an ``approximate'' factorization of $f$ in $\QQ(\alpha)$ with the $p$-adic norm. Then this factorization gives information about the absolute factorization.

\begin{proof}
Since $M(0)=0 \mod p$,  $0$ is also a root of $M_1(T)=\frac{M(T)}{\gcd(M(T),M'(T))}$ in $\FF_p$. As $p> \deg(M)$ we have $M_1'(0)\neq 0$ in $\FF_p$ and we can lift this root in $\QQ_p$ by Hensel's liftings. This gives a root of $M_1(T)$ in $\QQ_p$, thus a root of $M(T)$ in $\QQ_p$.
\end{proof}


\subsection{Choice of $p$}
\begin{lem}\label{evite_B}
Let $f(X,Y) \in \ZZ[X,Y]$, $\deg f(X,Y)\geq 1$ and let  $\mathcal B$ be a positive integer. There exist $(x_0,y_0) \in \ZZ^2$ and $p \in \ZZ$ such that $p$ divides $f(x_0,y_0)$ and $p$ does not divide $\mathcal  B$.
\end{lem}

\begin{proof}
We can reduce to the case of one variable and use the classical argument of  Dirichlet for proving that the set of prime numbers is infinite.

Consider the polynomial $f(X) \in \ZZ[X]$, $\deg f\geq 1$. Consider $x_1$ such that the constant term $c:=f(x_1)$ is not zero. 

Set $\tilde{f}(X)=f(X-x_1)$, so $c$ is the constant term of $\tilde f(X)$. Consider $\tilde{f}(c\mathcal BX)=c(1+\mathcal BXq(X))$, where $q(X)\in \ZZ[X]$ is not zero (otherwise $deg f<1$). We can find $x_0 \in \ZZ$, $x_0\neq 0$ such that $\mathcal Bx_0q(x_0)\neq 0$. Then, a prime $p$ dividing $1+\mathcal Bx_0q(x_0)$ does not divide $\mathcal B$ and we are done.
\end{proof}

\begin{defn}
 We say that the prime integer $p$ gives a {\em bad reduction} of $f(X,Y)$ if the number of absolute factors of $f(X,Y)\mod p$ differs from the number of absolute factors of $f(X,Y)$. \qed
\end{defn}

\begin{prop}\label{noether}
 Let $f(X,Y)$ be a rationally irreducible polynomial, monic in $Y$. Then there is a finite number of prime integers $p$ giving a bad reduction of $f(X,Y)$.

Furthermore, if $d(X)=disc_Y(f(X,Y))$, $d_1(X)=\text{ square-free part of }d(X)$ and $D=disc_X(d_1(X))$, the set of prime integers $p$ giving a bad reduction of $f$ is contained in the set of prime divisors of $D$. 
\end{prop}

\begin{proof}
The finiteness of the set of $p$ giving bad reductions comes from a theorem of \citet{Noe}.
For the characterization using $D$, we can say with other words that $f(X,Y)$ has a good reduction $\mathrm{mod}\, p$ if $d(X)$ and $d(X)\mod p$ have the same number of distinct roots. For the proof of this fact, see \citet{Trager}. Finally, for another proof, see \citet{Za}.
\end{proof}

\subsection{Recognition strategy}\label{approxp}

We assume that we chose a good prime $p$, such that $\tdeg(f)=\tdeg(f \mod p)$ and $f \mod p$ factors as $f(X,Y)=F^{(1)}(X,Y)\cdot G^{(1)}(X,Y) \mod p$ where $F^{(1)}$ is exactly the image $\mod p$ of an absolute factor $f_1$ of $f$.

In order to find the splitting field of $f(x_0,Y)$, relying on Proposition \ref{cst_primitif}, we need to compute $q(T)$, the minimal polynomial with integer coefficients of $\alpha:=f_1(x_0,y_0)$.

Starting from a factorization $f(x_0,Y)=F^{(1)}(x_0,Y)G^{(1)}(x_0,Y) \mod p$, we lift it through Hensel Lifting to the level of accuracy $p^\lambda$. We then consider the $p$-adic approximation $\overline \alpha:=F^{(\lambda)}(x_0,y_0)$ of $\alpha$. Using a ``big enough'' level of accuracy $\lambda$, we can compute the minimal polynomial of $\alpha$ from $\overline \alpha$.

\begin{prop}\label{LLLp}
 Consider $\overline \alpha=F^{(\lambda)}(x_0,y_0)$, $0\leq \overline \alpha\leq p^\lambda-1$ constructed above,  a positive integer $Q$ bounding the size of the coefficients of $q(T)$, $Q\geq \Vert q(T)\Vert_\infty$, and a positive  integer $\lambda\geq \log_p(2^{s^2/2}(s+1)^sQ^{2s})$.

Then we can compute the minimal polynomial $q(T)$ of $\alpha$ using the $LLL$ algorithm on an integer lattice whose basis is given using $\overline \alpha$ and $p^\lambda$.
\end{prop}

\begin{proof}
We apply the same construction of \citet[Section 16.4]{VZG} for detecting rational factors of univariate polynomials.

We consider the polynomials 
\[
\{T^i(T-\overline\alpha)|i=0,\dots,s-1\}\cup\{p^\lambda\}.
\]

We write as usual 
\[
T^i(T-\overline\alpha)=T^{i+1}-\overline\alpha T^i=\sum_{j=0}^s t_j T^j,
\]
where, in this case, $t_j\neq 0$ for $j\in\{i+1,i\}$ and $t_j=0$ otherwise. Then the associated vector for the polynomial $T^{i}(T-\overline\alpha)$ is
\[
b_i=(t_s,\dots,t_0).
\]

For the constant polynomial $p^\lambda$, we associate the vector $\tilde b=(0,\dots,0,p^\lambda)$.
We can construct the $(s+1)\times(s+1)$  matrix $B$ whose columns are the $b_i$, $i=0,\dots,s-1$ and $\tilde b$:

\begin{equation*}
B=
\begin{bmatrix}
1 & 0 & 0 & 0 & \dots & 0 & 0 & 0 \\
-\overline \alpha & 1 & 0 & 0 & \dots & 0 & 0 & 0 \\
0 & -\overline\alpha & 1 & 0 & \dots & 0 & 0 & 0 \\
0 & 0 & -\overline\alpha & 1 & \dots & 0 & 0 & 0 \\
\vdots &\vdots & \vdots & \vdots & \  & \vdots &\vdots &\vdots\\
0 & 0&0&0& \dots & 1 &0 & 0\\
 0 & 0&0&0& \dots & -\overline\alpha &1 & 0\\
 0 & 0&0&0& \dots & 0 &-\overline \alpha& p^\lambda
\end{bmatrix} 
\end{equation*}

If we consider a point $g$ of the integer lattice $\bigwedge(B)\subseteq \RR^{s+1}$ generated by the columns of the  matrix $B$, we can write its components with respect to the standard basis of $\RR^{s+1}$
\[
g=\sum_{i=0}^{s-1}g_ib_i+ \tilde g \tilde b=(g_{s-1}, g_{s-2}-\overline \alpha g_{s-1},\dots, g_{0}-\overline \alpha g_{1}, \tilde g p^\lambda-\overline \alpha g_0)
\]
and  associate a polynomial:
\[
G(T)=g_{s-1}T^s+(g_{s-2}-\overline \alpha g_{s-1})T^{s-1}+\dots+ (g_{0}-\overline \alpha g_{1})T+ \tilde g p^\lambda-\overline \alpha g_0=
\]
\[
=S(T)(T-\overline\alpha)+\tilde g p^\lambda \quad \text{ with } \quad S(T)=\sum_{i=0}^{s-1}g_iT^i.
\]
So if $g\in \bigwedge(B)$, the associated polynomial $G(T)$ has degree $\leq s$ and it is divisible by $(T-\overline\alpha)$ modulo $p^\gamma$.

The {\it vice versa} holds:\\
If $G(T)$ is a polynomial of degree at most $s$ and $G(T) \mod p^\lambda$ is divisible by $(T-\overline\alpha)$, then we can write
\[
 G(T)=S^\ast(T)(T-\overline \alpha)+R^\ast(T)p^\gamma \quad\text { with } \deg S^\ast(T)\leq s-1 \text{ and } \deg R^\ast(T)\leq s.
\]
 Using Euclidean division, we obtain $R^\ast(T)=S^{\ast \ast}(T)(T-\overline\alpha)+R p^\gamma$ with $\deg S^{\ast\ast}\leq s-1$ and $R$ a costant.
We define $S(T):=S^\ast(T)+p^\gamma S^{\ast\ast}(T)$. We then have that
\[
 G(T)=S(T)(T-\overline \alpha)+Rp^\gamma,
\]
that is, $G(T)$ can be written as a point of the lattice $\bigwedge(B)$.

So if we consider the matrix $B$ and we apply the LLL algorithm, we obtain as first vector of the reduced basis a ``short''vector representing a polynomial $G(T)$ with ``small'' norm such that $G(T)$ has degree $s$ and $G(T)\mod p^\lambda$ is divisible by $(T-\overline\alpha)$. Using the hypothesis $\lambda\geq \log_p(2^{s^2/2}(s+1)^sQ^{2s})$ we can apply  \citet[Lemma 16.20]{VZG}: we then have that $q(T)$ and $G(T)$ have a non-constant $\gcd$. But since $q(T)$ is irreducible and  $\deg q(T)=\deg G(T)$, we have that $q(T)=G(T)$.
\end{proof}


\medskip

To establish the level of accuracy $\lambda$, we need a bound on the size of the coefficients of the minimal polynomial of $\alpha$, $q(T)$.
Remember that
\[
q(T)=\prod_{i=1}^s(T-\alpha_i)=T^s+\sigma_1(\tilde\alpha)+\cdots +\sigma_{s-1}(\tilde\alpha)T+\sigma_s(\tilde\alpha),
\]
where $\sigma_i(\tilde\alpha)$ is the $i$-th symmetric function in the $\alpha=\alpha_1,\alpha_2,\dots,\alpha_s$.

Observe that
\[
 \vert \sigma_k(\tilde\alpha)\vert\leq \sum_{\tau\in \mathcal S_k}\vert\alpha_{\tau(1)}\vert\cdots \vert\alpha_{\tau(k)}\vert\leq \sum_{\tau   \in \mathcal{S}_k} \prod_{j=1}^m \vert y_j^{\tau(1)}\vert\cdots \prod_{j=1}^m \vert y_j^{\tau(k)}\vert,
\]
where $f_l(x_0,Y)=\prod_{j=1}^m(Y-y_j^{(l)})$ and $f(x_0,Y)=\prod_{i=1}^s f_l(x_0,Y)$.

As a bound on the coefficients of $f(x_0,Y)$ gives a bound on the $y_j^{(l)}$ \citep{VZG}, a bound on the coefficients of $f(x_0,Y)$ gives a bound for $\Vert q(T)\Vert_\infty$.

In practice, for ``early detection'', we rely on Proposition \ref{LLLp} replacing $Q$ by
 \[
 Q_1=\Vert f(x_0,Y)\Vert_\infty.
 \]

\begin{rem}\label{strat}
If $f(X,Y)$ is not monic, then we have to face two pro\-blems:
\begin{enumerate}
 \item \emph{Leading coefficient problem}: we cannot apply Hensel lifting in its ``classical'' form, because we need to have a factorization $f(x_0,Y)=F^{(1)}(x_0,Y)G^{(1)}(x_0,Y) \mod p$ in which $F^{(1)}(x_0,Y)$ or $G^{(1)}(x_0,Y)$ is monic.

\item In practical use of this construction of the minimal polynomial of $\alpha$, we will avoid  to lift the factorization until the level $\lambda$ of Proposition \ref{LLLp} (this bound is usually very pessimistic). However, in this way we are not sure that the polynomial $G(T)$ is actually $q(T)$. We then need a quick method to check if we found a good candidate to define the field extension or if we have to lift the factorization to a higher level of accuracy.
\end{enumerate}

\medskip

Consider $f(x_0,Y)=\sum_{i=0}^n \phi_iY^i$.

For what concerns the \emph{leading coefficient problem}, we can simply consider the ``mo\-di\-fied'' linear Hensel Lifting \citep[Algorithm 6.1]{GCL}. In this way we can lift the factorization modulo $p$,  but the coefficients involved in the computations are bigger,  since actually we lift a factorization of $\phi_n\cdot f(x_0,Y)$, obtaining a factor that we call $\tilde F^{(\lambda)}(Y)$.

\medskip

For what concerns the second problem,  we have to understand how the roots of a factor of $f(x_0,Y)$ are in connection with the coefficients of $q(T)$ and $\tilde f_1(Y)$, that is the factor of $f(x_0,Y)$ that we obtain after the ``modified'' Hensel Lifting. We call $q_s$ the leading coefficient of the polynomial $q(T)$.

If $f_1(x_0,Y)$ is the  factor of $f(x_0,Y)$ we are looking for, then the product of its roots is simply $\beta:=(-1)^{\deg \tilde f_1(Y)} \tilde f_1(y_0)/\phi_n$.

Then the product of the conjugated of $\beta$ is simply $q(0)/q_s$, but this is also the product of all the roots of $f(x_0,Y)$. So we have the following relation $\frac{q(0)}{q_s}=(-1)^{s}\frac{f(x_0,y_0)}{\phi_n}$. 

When we apply the $LLL$ algorithm to $\bigwedge(B)$ we can then proceed as follows: if the obtained polynomial $G(T)$ satisfies
\begin{equation*}\label{recogn}\tag{$\star$}
 \frac{G(0)}{G_s}=(-1)^s \frac{f(x_0,y_0)}{\phi_n} \qquad \text{with $G_s$ leading coefficient of }G(T)
\end{equation*}
then we will try to factor $f(x_0,Y)$ in the algebraic extension defined by  $G(T)$, that is $\QQ[T]/G(T)$. If $G(T)$ does not satisfy (\ref{recogn}), then we have to rise the level of approximation of the Hensel lifting and then apply again $LLL$ to the new lattice and test again. 

In this way we have a necessary condition that can help us to recognize the minimal polynomial of $\alpha$.

\end{rem}


\section{Absolute factorization algorithm}

We use the results and methods of the previous section to compute an absolute factor $f_1$ of $f$ (i.e. a representation of the field $\LL$ of its coefficient and the coefficients).
%
%
%
%

To ease the presentation, we rely on the practical evidence that for  random integer value $x_0$, $f(x_0,Y)$ is irreducible. In Section \ref{hilbert} we will present a variant using a weaker condition.


\medskip

\textbf{\textsc{Abs-Fac  algorithm}}\\
\textsc{Input:} $f(X,Y) \in \ZZ[X,Y]$, irreducible in $\QQ[X,Y]$ of degree $n$, a finite subset $S$ of $\ZZ^2$.\\
\textsc{Output:}  $q(T) \in \QQ[T]$ minimal polynomial of $\alpha$ defining the minimal algebraic extension $\LL=\QQ(\alpha)=\QQ[T]/q(T)$ and $f_1(X,Y) \in \LL[X,Y]$ an absolute irreducible factor of $f$, or ``I don't know''.\\
\textsc{Preprocessing:} Choose $(x_0,y_0) \in S^2$, such that $f(x_0,Y)$ is irreducible. If all of the points were used, then return ``I don't know''.

\begin{enumerate}

\item \label{stepchoice}Choose a prime $p$ dividing $f(x_0,y_0)$ such that $\tdeg(f \mod p)= \tdeg(f)$. 

\item\label{facto_dans_algo} Factorize $f$ in $\FF_p[X,Y]$.

If $f\mod p$ is irreducible and satisfies an absolute irreducibility test then Return ``$f$ is absolutely irreducible'', $f_1:=f$ and $q(T):=T$.

If $f \mod p$ is irreducible and not absolutely irreducible  then go to the Preprocessing step (choosing a point $(x_0,y_0)$ not yet used and a different prime $p$).

Else ${f}(X,Y)=F^{(1)}(X,Y)\cdot G^{(1)}(X,Y) \mod p$ where $F^{(1)}$ is one of  the irreducible factors in $\FF_p[X,Y]$ with smallest degree $m$, check that $s:=\frac{\tdeg (f)} {m}$ is an integer else go to the Preprocessing step (choosing a point $(x_0,y_0)$ not yet used and a different prime $p$). 

\item \label{stepalpha}Lift the factorization to ${f}(x_0,Y)=F^{(\lambda)}(x_0,Y)G^{(\lambda)}(x_0,Y) \mod p^\lambda$ ; $\lambda$ is chosen according to Proposition \ref{LLLp} and Remark \ref{strat}.

\item \label{steprecogn}Define $\overline\alpha:=F^{(\lambda)}(x_0,y_0) \in \ZZ/p^\lambda \ZZ$. Find, using the lattice described in Section \ref{approxp} and the LLL algorithm, the polynomial $q(T)$. If $q(T)$ does not satisfy (\ref{recogn}) or it is not irreducible, go back to step (\ref{stepalpha}) and double $\lambda$.
\item \label{factors} Denote by $\alpha$ a root of $q(T)$ (i.e. the command RootOf in Maple) then factorize $f(x_0,Y)$ in $\QQ(\alpha)[Y]=\LL[Y]$ and denote by $F_1(x_0,Y)$ a factor with degree $m$ and with $F_1(x_0,y_0)=\alpha$.\\
If we do not find such a factor, then go to the Preprocessing step (choosing a point $(x_0,y_0)$ not yet used and a different prime $p$). 

\item  \label{fin} Perform $m$ times $X$-adic Hensel liftings on $f(x_0,Y)=F_1(x_0,Y)F_2(x_0,Y)$ to determine a candidate for $f_1(X,Y)$ in $\LL[X,Y]$ and check that it divides $f(X,Y)$.
Else go to the Preprocessing step (choosing a point $(x_0,y_0)$ not yet used and a different prime $p$).

\medskip

Return $q(T)$ and $f_1(X,Y)$.
\end{enumerate}

\begin{prop}
The algorithm  gives a correct answer.
\end{prop}
\begin{proof}
Since it is a Las Vegas algorithm, this algorithm is probably fast and always correct but the answer can be ``I don't know''. So we just have to check that a given positive answer is correct.

The starting point of the proposed algorithm, as in the irreducibility test, is to determine a prime $p$ such that the reduction  modulo $p$ kills the evaluation of $f$ on an integer point $(x_0,y_0)$. Then the constant term of the minimal polynomial of $\alpha:=f_1(x_0,y_0)$  vanishes modulo $p$. Such a $p$ is easily found. However we rely on randomness to expect with a good probability that $\LL=\QQ(\alpha)$  and that $f$ has good reduction modulo $p$ (using  Proposition \ref{noether} and Lemma \ref{rootinQp}).

In the algorithm described above, we inserted  some checks and a loop to change $p$ if it is an ``unlucky'' choice.
The algorithm can be made deterministic (but less efficient) by considering a large testing set for $(x_0,y_0)$ and take $p$ not dividing a huge constant $\mathcal{B}$ computed a la Trager, to avoid bad reduction. We would be able to do this thanks to Lemma \ref{evite_B}.

The output of the algorithm, the factor $f_1$, is irreducible in $\LL[X,Y]$. Indeed, $f_1(x_0,Y)=F_1(x_0,Y)$ and $F_1(x_0,Y)$ is irreducible in $\LL[Y]$ because of the irreducibility of $f(x_0,Y)$ in the Preprocessing Step. Furthermore, the extension $\LL$ is minimal. Indeed, at the end of the algorithm we have $\deg_Y f_1=m$, $\deg q=s$ and $s.m=n$, see the definition of $s$ in Step \ref{facto_dans_algo}.
\end{proof}

Remark: $f_1$ is irreducible modulo $p$ and $f_1$ modulo $p$ generically satisfies condition $(C)$, so Proposition \ref{test} guaranties the absolute irreducibility of $f_1$ in $\LL[X,Y]$.

\subsection{Parallel version of the Algorithm}\label{parallVers}
 In step (\ref{factors})  of the Abs-Fac Algorithm we perform a factorization of $f(x_0,Y)$ in the polynomial ring $\LL[Y]$. Then in Step  (\ref{fin})  we use Hensel liftings to reconstruct the factor $f_1$. If we use parallel calculus in these steps, we can perform $(m+1)$ Lagrange interpolations to reconstruct the factor $f_1$. We have to assume that in the factorization of $f(x_0,Y)$ in $\LL[Y]$ there is only one factor of degree $m$. This is not always verified, for instance if the extension $\LL$ is normal we may have several factors of the same degree $m$.

We write the absolute factor $f_1$ as
\[
 f_1(X,Y)=Y^m+\sum_{k=0}^{m-1}\sum_{i+j=k} a_{i,j}^{(1)}X^i Y^j= Y^m+\sum_{j=0}^{m-1} b_{j}(\alpha,X) Y^j,
\]
where $b_j(Z,X)\in \QQ[Z,X]$ of degree $\leq m-j$ and $\alpha$ is a root of the polynomial $q(T)$ found in step (4).

We then want to find the polynomials $b_j(\alpha,X)$.

We substitute steps (\ref{factors}) and (\ref{fin}) with the following procedure:
\begin{itemize}
 \item[(5bis)]  Denote by $\alpha$ a root of $q(T)$ (i.e. the command RootOf in Maple).\\
 Choose  points $x_1,\dots,x_m\in \ZZ$, $x_i\neq x_0$ for $i=1,\dots,m$ such that $f(x_i,Y)$ is rationally irreducible.\\
Compute the factorization of $f(x_i, Y)$ in $\LL[Y]$ and choose $F_{1,0}(Y)$ from the factorization of $f(x_0,Y)$ as in step (5) of the algorithm and $F_{1,j}(Y)$ a factor of minimal degree $m$ in the factorization of $f(x_j,Y)$.
\item[(6bis)] Write $F_{1,j}(Y)$ as follows
\[
 F_{1,j}=\sum_{i=0}^m \gamma_{i,j}Y^j \text{ with }\gamma_j\in \LL.
\]
We then construct the polynomials $b_j(\alpha,X)$ of degree $j$ using Lagrange interpolation \citep[Section 3.1]{burden} on the set of nodes $\gamma_{0,j},\dots,\gamma_{j,j}$.
In this way we determine a candidate for $f_1(X,Y)$ in $\LL[X,Y]$. We check that it divides $f(X,Y)$.
Else go to the Preprocessing step (choosing a point $(x_0,y_0)$ not yet used and a different prime $p$).
\end{itemize}

\medskip

The advantage of steps (5bis) and (6bis) is that in this way this part of the algorithm can be naturally parallelized and do not saturate the memory.

\subsection{Hilbert's Irreducibility Theorem}\label{hilbert}

In the preprocessing step we check that $f(x_0,Y)$ is irreducible. This situation happens very often in practice. With a more theoretical point of view, we know that there exists an infinite number of $x_0 \in \ZZ$ such that $f(x_0,Y)$ is irreducible, thanks to Hilbert's irreducibility theorem. There  exists bounds for this theorem but unfortunately they are very big, see \citet{DebWal}. 

Here we now use a weaker condition on the choice of $(x_0,y_0)$ that allows us to reconstruct the factor $f_1(X,Y)$ even if $f(x_0,Y)$ is not rationally irreducible.

Choose an integer point $(x_0,y_0)\in \ZZ^2$  such that $x_0$ is not a root of the polynomial $\Delta(X)=\mathrm{disc}_Y(f(X,Y))$ and choose an integer $p$ such that $\Delta(x_0) \mod p\neq 0$.
 With this choice of $(x_0,y_0)$ we are sure that the univariate polynomial $f(x_0,Y)$ has no multiple roots in $\QQ$ nor in $\FF_p$.

We do not assume that $f(x_0,Y)$ is rationally irreducible. We computed the factorization $\mod p$
\[
 f(X,Y)=F(X,Y)\cdot G(X,Y) \in \FF_p[X,Y] \quad \deg F=m.
\]
Thanks to the choice of $p$ as in step (\ref{stepchoice}) of the algorithm, $F(X,Y)$ should be equal $\mod p$ to the researched absolute factor $f_1(X,Y)$ of $f$.

After applying step (\ref{factors}), we get the following factorization
\begin{equation}\label{nonirr}
 f(x_0,Y)=\psi_1(Y)\cdots \psi_r(Y)\in \QQ(\alpha)[Y]
\end{equation}
and need to find the set of indexes $I\subseteq\{1,\dots,r\}$ such that
\begin{equation}\label{nonirrf1}
 \prod_{i\in I}\psi_i(Y)=f_1(x_0,Y).
\end{equation}
We reduce $\mod p$ the equalities (\ref{nonirr}) and (\ref{nonirrf1}). We  obtain that $j \in I$ if and only if $\psi_j \mod p$ divides $F(x_0,Y) \mod p$.

\section{Examples and practical complexity}

We tested our algorithm on several examples, using (probably non-optimal) routines implemented in Maple 10.

We focused on the construction of the minimal polynomial $q(T)$ of $\alpha$, that is on the construction of the splitting field $\QQ(\alpha)$; in fact the last part of the algorithm ($X$-adic Hensel lifting or Lagrange interpolation)  depends strongly on the used software.

The procedures, data and Maple files of several examples are available at \\
http://math.unice.fr/$\sim$cbertone/ 

Here we list some remarks about  both the strong and the weak points of our algorithm arising from the computed examples.

 \noindent $\bullet$ In general the algorithm is quite fast: it took around 30 sec (factorization $\mod p$, Hensel lifting, construction of the minimal polynomial) to compute the polynomial $q(T)$ starting from a polynomial of degree 200, with 10 absolute factors of degree 20 each.

\noindent $\bullet$ If possible, it seems to be a good idea to choose a ''small'' prime $p$ (in this way we can gain some time in the $\mod p$-factorization). If  the integers dividing $f(x_0,y_0)$ are quite big, it may be better to go back to the preprocessing step.

\noindent $\bullet$ On examples of high degree, the most of the time is spent for the construction of the minimal polynomial from the approximation $\overline \alpha$. In our tests, we used the $LLL$ function of Maple, but we may speed up this part of the computation using more performing algorithms for $LLL$ (for example, see \citet{NS} and \citet{Sch}).

\noindent $\bullet$ For the computation of the $p$-adic Hensel Lifting, we have implemented a small procedure in Maple, both for the linear and the quadratic one, which can deal also with non-monic polynomials \citep[Algorithm 15.10]{VZG}.

\subsection*{Benchmark}

We consider random polynomials $g_1 \in \mathbb Q[x,y,z]$ and $g_2\in \mathbb Q[z]$, of degrees $d_1$ and $d_2$ resp. both rationally irreducible. We compute $f(X,Y)=Res_z(g_1,g_2)$. In this way we obtain an irreducible polynomial $f(X,Y)\in \mathbb Q[x,y]$, monic in $y$, of degree $d_1\cdot d_2$ with $d_2$ absolute irreducible factors each of degree $d_1$. 

\medskip

The polynomials $g_1$ and $g_2$ used are listed in the file ``Polynomials.mws''.

\medskip

Here we summarize the time needed to obtain $q(T)$, the minimal rational polynomial of $\alpha$, such that the absolute factors of $f(X,Y)$ are in $\LL[x,y]$, $\LL=\mathbb Q(\alpha)=\mathbb Q[T]/q(T)$ and we made a few remarks about the strategy one may adopt (for instance the choice of the prime).

In almost all of the examples, we compute the Hensel lifting both with the linear and the quadratic algorithm, this is why we always chose as level of accuracy a power of $2$.

In the first 2 examples, we also computed the factorization of $f(x_0,Y)$ in $\QQ(\alpha)$.

In the first example, we computed the factor $f_1(X,Y)$ using Lagrange Interpolation

To repeat the examples, one need to change at the beginning of each Maple file the location of the file ``proc.txt'', in which there are (non-optimal) implementations for linear and quadratic Hensel Lifting (for non monic polynomials) and a procedure to compute the minimal polynomial of a $p$-adic approximation of $\alpha$ using the $LLL$ algorithm.

The names of kind ``Example1.2.mws'' refer to the Maple files on the website.

\begin{exmp}

$f(X,Y)$ rational irreducible polynomial of degree 50 with 5 absolute factors of degree 10.

We need 1.5 sec to construct the example and factor $f(0,0)$. We construct the minimal polynomial defining the field extension for 2 different choices of $p$.

\medskip

Example1.1.mws: we choose $p=11$.

$\bullet$ Time to factor $f(X,Y) \mod\, p$: $0.131$ sec.

The estimation of the level of accuracy that ensures the correct computation of $q(T)$ is in this case 338; we choose to lift the factorization to the level $p^{256}$.

$\bullet$ Time to lift the factorization $f(0,Y)=g_1(0,Y)g_2(0,Y) \mod\, p$ to a fac\-to\-ri\-za\-tion $\mod p^{256}$, using:\\
Linear Hensel Lifting: less than $1$ sec\\
 Quadratic Hensel Lifting: less than $0.07$ sec.

$\bullet$ Time to find the minimal polynomial of $\alpha$ through its  approximation $ \mod p^{256}$ using $LLL$: $0.22$ sec. 

\medskip

We can complete the algorithm using steps (5bis) and (6bis):\\
we choose 10  nodes $x_1\dots,x_{10}$ randomly and factor the polynomials $f(x_j,Y)$ in $\QQ(\alpha)[Y]$; the longest of these factorization takes  about $219$ sec. Then we use Lagrange Interpolation and obtain $f_1(X,Y)$.\\

Example1.3.mws: if we use the software \emph{Pari GP}, applying the function \emph{polred()} to the obtained polynomial $q(T)$, we get $q_1(Z)$ which defines the same algebraic extension as $q(T)$ but has smaller coefficients. In this way, the factorization of $f(0,Y)$ in $\QQ(\alpha)$ took only 8 sec, but the computation of the polynomial $q_1(Z)$ in Pari GP took more than 360 sec! \qed

\end{exmp}

\begin{exmp}

$f(X,Y)$ rational irreducible polynomial of degree 400 with 20 absolute factors of degree 20.

We need around $1260$ sec to construct the example and factor $f(0,0)$. 

\medskip

Example6.1.mws: we choose $p=53259165137$.

$\bullet$ Time to factor $f(X,Y) \mod p$: $1924$ sec.

The estimation of the level of accuracy that ensures the correct computation of $q(T)$ is in this case 398; we choose to lift the factorization to the level $p^{256}$.

$\bullet$ Time to lift the factorization $f(0,Y)=g_1(0,Y)g_2(0,Y) \mod p$ to a fac\-to\-ri\-za\-tion $\mod p^{256}$, using \\
Linear Hensel Lifting: less than $365$ sec\\
 Quadratic Hensel Lifting: less than  $39$ sec.

$\bullet$ Time to find the minimal polynomial of $\alpha$ through its  approximation $ \mod p^{256}$ using $LLL$: $1024$ sec. 

\medskip
In order to compare the time needed for the construction of $q(T)$ computing modulo a ``small'' prime, we considered also the case with $p=89$ dividing $f(-1,0)$. In this case we obtained (Example6.2.mws):

$\bullet$ Time to factor $f(X,Y) \mod p$: $127$ sec.

The estimation of the level of accuracy that ensures the correct computation of $q(T)$ is in this case 2194; we choose to lift the factorization to the level $p^{1024}$.

$\bullet$ Time to lift the factorization $f(0,Y)=g_1(0,Y)g_2(0,Y) \mod p$ to a fac\-to\-ri\-za\-tion $\mod p^{1024}$, using \\
Linear Hensel Lifting:  $737$ sec\\
 Quadratic Hensel Lifting:  $24$ sec.

$\bullet$ Time to find the minimal polynomial of $\alpha$ through its  approximation $ \mod p^{1024}$ using $LLL$: $520$ sec. \qed 

\end{exmp}

For the detail of other examples, see http://math.unice.fr/$\sim$cbertone/

In the following table we resume the timings of a few  more examples.

\begin{itemize}
 \item  $n=\tdeg(f)$, $s$=number of absolute factors of $f$, $m=n/s$=degree of an absolute factor of $f$;
\item $p$= prime integer, $\lambda=$ level of accuracy of Proposition \ref{LLLp}, $\tilde\lambda=$chosen level of accuracy;
\item $T_1=$ time to factor $f(X,Y) \mod p$, $T_2=$time to lift the factorization to $p^{\tilde\lambda}$, $T_3=$time to find the minimal polynomial of $\alpha$.
\end{itemize}

\medskip

\begin{center}
\begin{tabular}{|c|c|c|c||c|c|c||c|c|c|}
\hline
$Example$ & $n$ & $s$  & $m$ & $p$ &  $\lambda$ & $\tilde \lambda$ & $T_1$ & $T_2$ & $T_3$\\
\hline
\hline
Example 1.1 & 50 & 5 & 10 & 11 & 338 & 256 & 0.13 s & 0.07 s & 0.22 s\\
\hline
Example 1.2 & 50 & 5 & 10  & 307 & 141 & 128 & 0.13 s & 0.08 s & 0.4 s\\
\hline
Example 2.1 & 100 & 10 & 10 & 7 & 1105 & 512 & 3.4 s & 0.3 s & 2.25 s\\
\hline
Example 2.2 & 100 & 10 & 10 & 655379 & 160 & 128 & 6.2 & 0.4 s & 5.7 s\\
\hline
Example 3.1 & 150 & 15 & 10 & 7 & 2246 & 1024 & 10 s & 1.08 s & 21 s \\
\hline
Example 4.1 & 200 & 10 & 20 & 47 & 853 & 512 & 33 s & 2.8 s & 14 s\\
\hline
Example 4.2 & 200 & 10 & 20 & 114041 & 282 & 256 & 128 s & 3.8 s & 30 s\\
\hline
Example 5 & 200 & 20 & 10 & 7682833 & 457 & 256 & 68 s & 3.8 s & 220 s\\
\hline
Example 6.1 & 400 & 20 & 20 & 53259165137 & 398 & 256& 1924 s & 39 s & 1024 s\\
\hline
Example 6.2 & 400 & 20 & 20 & 127 & 2194 & 1024 & 127 s & 24 s & 520 s\\
\hline
Example 7 & 100 & 20 & 5 & 7 & 3029 & 2048 & 0.64 s & 1.25 s & 205 s\\
\hline
\end{tabular}\\
\end{center}

\section{Conclusion}

In this paper we have presented  a new approach to absolute factorization improving the use of classical tools, in particular the TKTD algorithm and $LLL$ algorithm.

In fact, we have refined the main idea of the TKTD algorithm (\citet{TrDv}, \citet{K1}, \citet{Trager2}), because we  construct a ``small'' algebraic extension field in which the polynomial $f(X,Y)$ splits. However the degree of the extension constructed by our algorithm is minimal, i. e. the number of absolute factors. In the TKTD algorithm the degree of the used extension is the degree of the polynomial $f(X,Y)$.

Furthermore, we use the $LLL$ algorithm in a new way to define the field extension, while its classical applications are on the coefficients of a univariate rational polynomial in order to factor it \citep{LLL}, or on the exponents (see \citet{vanH} and \citet{Che}).

In our application, $LLL$ is used on a lattice defined by $s+1$ vectors, where $s$ is the number of absolute factors of the polynomial, which is  smaller than the degree of the polynomial to factor. That is why in our algorithm the use of $LLL$ is not a bottleneck. 

Nevertheless, we may improve the fastness of the computations using, if it will be available in the future,  a fast $LLL$ (see \citet{NS} and \citet{Sch}) and a good implementation of the Polred algorithm \citep{Cohen}, which allows a better presentation of the algebraic field extension.

Our Maple prototype was able to deal  with high degree polynomials (up to 400), which were so far out of reach of all other absolute factorization algorithm; furthermore it is very fast on polynomials of middle degrees (about 100).

An efficient implementation of our algorithm will also need  good $p$-adic and $X$-adic Hensel liftings. We expect, in a near future, that the library Mathemagix \citep{Mathemagix}  will provide  optimized implementations of these routines. Another point to improve is the parallel version of the algorithm, in order to be able to deal also with normal extensions of $\QQ$.

Another related direction of research that we will soon explore, is extending some of these techniques to the decomposition of affine curves in dimension 3 or more.

\bigskip

\bigskip

\section*{Acknowledgments}
The authors would like to thank Gr\'egoire Lecerf for useful discussions and valuable suggestions concerning this paper.

\bibliographystyle{elsart-harv}
\bibliography{biblio}


\end{document}